\documentclass[a4paper,12pt,reqno]{amsart}
\usepackage{marginnote}
\usepackage{amsmath}
\usepackage{mathtools}
\usepackage{amsfonts}
\usepackage{amscd}
\usepackage{amssymb}
\usepackage{graphicx}
\usepackage{mathrsfs}
\usepackage{dsfont}
\usepackage{tikz}

\usepackage[dvips,all,arc,curve,color,frame]{xy}
\usepackage{enumitem}

\usepackage[colorlinks]{hyperref}
\usepackage[nameinlink]{cleveref}
\usepackage{tikz,mathrsfs}
\usetikzlibrary{arrows,decorations.pathmorphing,decorations.pathreplacing,positioning,shapes.geometric,shapes.misc,decorations.markings,decorations.fractals,calc,patterns}

\newtheorem{thm}{Theorem}[section]

\newtheorem{lemma}[thm]{Lemma}
\newtheorem{claim}[thm]{Claim}

\theoremstyle{definition}

\begin{document}
\author{Vojt\v{e}ch Dvo\v{r}\'ak}
\address[Vojt\v{e}ch Dvo\v{r}\'ak]{Department of Pure Maths and Mathematical Statistics, University of Cambridge, UK}
\email[Vojt\v{e}ch Dvo\v{r}\'ak]{vd273@cam.ac.uk}

\title[Waiter-Client Clique-Factor Game]{Waiter-Client Clique-Factor Game
}


\begin{abstract} 
Fix two integers $n, k$, with $n$ divisible by $k$, and consider the following game played by two players, Waiter and Client, on the edges of $K_n$. Starting with all the edges marked as unclaimed, in each round, Waiter picks two yet unclaimed edges. Client then chooses one of these edges to be added to Client's graph, while the other edge is added to Waiter's graph. Waiter wins if she eventually forces Client to create a $K_k$-factor in Client's graph. If she does not manage to do that, Client wins. 

For fixed $k$ and large enough $n$, it can be easily shown that Waiter wins if she plays optimally (in particular, this is an immediate consequence of our result that for such $n$, Waiter can win quite fast). The question posed by Clemens et al. is how long the game will last if Waiter aims to win as fast as she can, Client tries to delay her as much as he can, and they both play optimally. We denote this optimal number of rounds by $\tau_{WC}(\mathcal{F}_{n,K_k-\text{fac}},1 )  $. 
In the present paper, we obtain the first non-trivial lower bound on this quantity for large $k$. Together with a simple upper bound following the strategy of
Clemens et al., this gives $$2^{k/3-o(k)}n  \leq  \tau_{WC}(\mathcal{F}_{n,K_k-\text{fac}},1 ) \leq 2^k\frac{n}{k}+C(k),$$
where $C(k)$ is a constant dependent only on $k$ and the $o(k)$ term is independent of $n$ as well.
\end{abstract}

\maketitle

\section{Introduction}



Combinatorial games form a broad and widely studied field. In the most usual setting, we have two players and a set of rules and we try to determine which player has a winning strategy if both players play optimally. Commonly, even for simple looking games, this may be a very difficult question. To the reader interested in the topic in general, we recommend the book of Beck \cite{beck2008combinatorial}. The book of Hefetz, Krivelevich, Stojakovi{\'c} and Szab{\'o} \cite{hefetz2014positional} about the positional games (which form a large part of the well studied combinatorial games) is also a good reference.

One of the well known classes of the combinatorial games are the so-called \textit{Waiter-Client games}, originally introduced by Beck \cite{beck2002positional} under the name Picker-Chooser games. These are played by two players, Waiter and Client. Initially, we are given an integer $b \geq 1$, a finite set $X$ and a family of winning sets $\mathcal{F}$ of the subsets of $X$. All elements of $X$ are initially marked as unclaimed. In each round, Waiter chooses $b+1$ yet unclaimed elements of $X$ and offers them to Client. Client picks one of them and adds it to his graph, while the remaining $b$ elements are added to Waiter's graph. Waiter wins if she forces Client to create a winning set $F \in \mathcal{F}$ in Client's graph, and otherwise if Client can prevent Waiter from doing so, Client wins.

Various papers studied, for given $X,b,\mathcal{F}$, which player has a winning strategy in the corresponding Waiter-Client games. Or, in the cases when we know that Waiter can win, how fast can she guarantee her victory to be. For instance, Bednarska-Bzdega, Hefetz, Krivelevich and {\L}uczak \cite{bednarska2016manipulative} studied such games with the winning sets being large components or long cycles. Hefetz, Krivelevich and Tan \cite{hefetz2016waiter} looked on the Waiter-Client games involving planarity, colourability and minors, and later \cite{hefetz2017waiter} the same authors studied a Hamiltonicity game with a board being a random graph. Krivelevich and Trumer \cite{krivelevich2018waiter} considered a maximum degree game. Yet more results were obtained by Tan \cite{tan2017waiter} about colourability and $k$-SAT games. 



When $b=1$, we call the corresponding Waiter-Client game \textit{unbiased}. Clemens et al. \cite{clemens2020fast} studied several unbiased Waiter-Client games played on the edges of the complete graph, i.e. with $X=E(K_n)$.

Assume that for our triple $X,b,\mathcal{F}$, Waiter wins the corresponding Waiter-Client game. Then we will denote by $\tau_{WC}(\mathcal{F},b)$ the number of rounds that the game will last when Waiter tries to win as fast as possible, Client tries to slow her down as much as possible, and they both play optimally. Note that $X$ does not appear in this notation - but that will not be an issue, as from now on, we will only consider $X=E(K_n)$. 

One of the games considered by Clemens et al. is the so-called \textit{triangle-factor game}, where the winning sets are triangle-factors, i.e. decompositions of the vertices of $K_n$ into disjoint triples, each of which forms a triangle $K_3$ (hence we must in particular insist that $n$ is divisible by $3$). It is known that for large enough such $n$, Waiter has a winning strategy here. Clemens et al. \cite{clemens2020fast} moreover proved that for such large $n$ divisible by $3$, we have $\frac{13}{12}n \leq \tau_{WC}(\mathcal{F}_{n,K_3-\text{fac}},1 ) \leq \frac{7}{6}n+o(n)$. The lower bound was later improved by the present author \cite{dvovrak2021waiter}, so that the precise value of $\tau_{WC}(\mathcal{F}_{n,K_3-\text{fac}},1 )$ is now known up to the $o(n)$ additive term.

\begin{thm}\label{triangleresult}
For $n$ divisible by $3$ and large enough that Waiter wins the triangle-factor game, we have $$(\mathcal{F}_{n,K_3-\text{fac}},1 ) = \frac{7}{6}n+o(n).$$
\end{thm}

Clemens et al. pose a question what happens in the more general case of the \textit{$k$-clique-factor game} for large values of $k$. Here, we insist that $n$, the number of vertices of our graph, is divisible by $k$, and the winning sets are $K_k$-factors, i.e. decompositions of the vertices of $K_n$ into disjoint sets of size $k$, each of which forms a clique $K_k$. Once again, it is easy to see that for $n$ large enough in terms of $k$, Waiter has a winning strategy.

Addressing this question of Clemens et al., our main aim in this paper is to give the first non-trivial lower bound on $\tau_{WC}(\mathcal{F}_{n,K_k-\text{fac}},1 )$. Combined with a simple upper bound (which we expect to be of the correct magnitude) following the strategy of Clemens et al. for the triangle-factor game, this gives the following result.

\begin{thm}\label{wckthm}
There exist functions $n_0(k), C(k)$ such that one has $$ 2^{k/3-o(k)}n  \leq  \tau_{WC}(\mathcal{F}_{n,K_k-\text{fac}},1 ) \leq 2^k\frac{n}{k}+C(k)$$ for $n \geq n_0(k)$ and divisible by $k$.
\end{thm}

Most of the paper deals with proving the more difficult lower bound. The strategy of Client that we use is a trivial random one - Client always picks either element with equal probability and independently of the other rounds. Once we show that with a positive probability Client can survive a certain number of rounds (no matter what strategy Waiter uses), that guarantees a deterministic strategy of Client that ensures he can always survive that many rounds.


The difficult part of our approach is finding a quantity which we can bound in the expected value when Client plays randomly, no matter what strategy Waiter uses, and suitably relate to a $K_k$-factor. Such a quantity cannot be simple like the number of vertices that are part of some $k$-clique in Client's graph. Indeed, Waiter can initially create a clique $A$ with $2k-2$ vertices in Client's graph, and then one by one make each vertex $v$ in the graph part of a $k$-clique consisting of $v$ and $k-1$ elements of $A$, while using only about $kn$ rounds in total. So one needs a lot more delicate definition of events whose probabilities we bound.

The structure of the paper is as follows. To motivate our main proof, in section \ref{basicbound} we present a simpler and easier to grasp lower bound using similar ideas. Building on this, we prove the lower bound in Theorem \ref{wckthm} in section \ref{imprbound}. In section \ref{secupbound}, we use a strategy of Clemens et al. to obtain the upper bound in Theorem \ref{wckthm}. And in section \ref{secconclu}, we give some final remarks.

\section{A simple exponential bound}\label{basicbound}

In this section, we will motivate the proof of our main result by providing a simple proof along similar lines of a somewhat weaker lower bound (which still has a term exponential in $k$ in front of $n$).

\begin{thm}\label{weakerthm}
Fix $k \geq 100$ and consider $n$ divisible by $k$ and large enough that Waiter has a winning strategy for the corresponding Waiter-Client $k$-clique-factor game.
Then $\tau_{WC}(\mathcal{F}_{n,K_k-\text{fac}},1 ) \geq 2^{k/6-\frac{k}{\log k}}n$.
\end{thm}


The strategy of Client will be random and extremely simple - Client always takes the two edges offered and chooses the one to add to his graph randomly, with equal probabilities for each. If we can show that with probability at least $1/2$, no matter what strategy Waiter uses, she does not win against randomly playing Client after $2^{k/6-\frac{k}{\log k}} \, n$ rounds, then that also implies Client has a deterministic strategy guaranteeing him to make the game last at least that long. Note that the choice of the probability $1/2$ for our purpose is arbitrary - we could replace it with any constant $p$ such that $0<p<1$.

We shall colour the edges of Client's graph red and the edges of Waiter's graph blue. For a given vertex $v$, we shall also denote by $N_R(v)$ its red neighbourhood and by $N_B(v)$ its blue neighbourhood.

The proof is based on an idea that after $2^{k/6-\frac{k}{\log k}}n$ rounds, if there was a red $k$-clique factor already created, many of the cliques $C$ in this $k$-clique factor will contain some vertex $z_i$ such that, if we denote the vertices of $C$ by $z_1,...,z_k$:

\begin{itemize}
    \item $z_i$ has red degree less than $2^{k/6-\frac{k}{2\log k}}$;
    \item for many pairs $z_j,z_l$, both the red edges $z_i z_j$ and $z_i z_l$ were added before the red edge $z_j z_l$ (here many means $(\frac{1}{6}-o(1))k^2$). 
\end{itemize}

We then bound the probability of this event happening at any vertex $z_i$ and conclude by taking expectation that on average, there will not be many such vertices, which gives the desired contradiction.
In the next few pages we shall formalize this idea and make it work.

Assume that $k \geq 100$, that $n$ is divisible by $k$, that $n$ is large enough that Waiter has a winning strategy for the corresponding Waiter-Client $K_k$-factor game and that $2^{k/6-\frac{k}{\log k}} \, n$ rounds have passed. If Waiter already won at this point, there are disjoint red $k$-cliques $C_1,...,C_{n/k}$ partitioning the vertices of the graph.

Call a vertex a \textit{high degree vertex} if it has red degree at least $2^{k/6-\frac{k}{2\log k}}$, and call it a \textit{low degree vertex} otherwise. Call a clique a \textit{high degree clique} if it contains at least one high degree vertex, and call it a \textit{low degree clique} otherwise.

\begin{claim}\label{initialclaim}
Among $C_1,...,C_{n/k}$, there are at most $\frac{n}{2k}$ high degree cliques.
\end{claim}

\begin{proof}
Assume for a contradiction that this is false, so we have at least $\frac{n}{2k}$ high degree vertices. As every red edge touches at most two high degree vertices (i.e. its two endpoints), that implies we have at least $ 2^{k/6-\frac{k}{2\log k}} \, \frac{n}{4k}$ red edges in our graph. Since only $2^{k/6-\frac{k}{\log k}}n$ rounds have passed, we have $$2^{k/6-\frac{k}{\log k}}n<2^{k/6-\frac{k}{2\log k}} \, \frac{n}{4k},$$
for $k \geq 100$, which is a contradiction.
\end{proof}

Next, we need to define several events. For a given vertex $v$, denote by $X(v)$ the event that $v$ is not a high degree vertex after $2^{k/6-\frac{k}{\log k}} \, n$ rounds have passed. Denote by $Y(v)$ the event that:

\begin{itemize}
    \item there exists a subset $w_1,...,w_{k-1}$ in $N_R(v)$ after $2^{k/6-\frac{k}{\log k}} \, n$ rounds have passed that forms a red clique;
    \item moreover, at least $\frac{(k-1)(k-2)}{6}$ edges $w_i w_j$ within this subset were added after both edges $v w_i$ and $v w_j$ have been added.
\end{itemize}

Finally let $S(v)=X(v) \cap Y(v)$.

The following simple lemma explains why we care about the event $Y(v)$ - it is because having a red $K_k$-factor, we can guarantee this event to happen for many vertices.

\begin{lemma}\label{thirdsees}
Start with an empty graph on $k$ vertices and keep adding edges in some order until our graph is complete. Then there exists a vertex $v$ such that at least $\frac{(k-1)(k-2)}{6}$ edges $w_i w_j$ were added after both the edges $v w_i$ and $v w_j$ were added.
\end{lemma}

It is not hard to see that up to the smaller order terms, this result is best possible - indeed, just adding the edges in random order, it will typically happen that for each vertex $v$, about $\frac{k^2}{6}$ edges $w_i w_j$ will be added after both the edges $v w_i$ and $v w_j$ were added.

\begin{proof}[Proof of Lemma \ref{thirdsees}]
Call a pair $(v,w_i w_j)$ consisting of a vertex $v$ and an edge $w_i w_j$ good if $w_i w_j$ is the last edge added to triangle $v w_i w_j$. Clearly, there are ${ k \choose 3 }$ good pairs, so some vertex $v$ appears in at least $\frac{1}{k} {k \choose 3}=\frac{(k-1)(k-2)}{6} $ good pairs.



\end{proof}

Combining Claim \ref{initialclaim} and Lemma \ref{thirdsees}, we can conclude that if Waiter won fast, there in fact must be many vertices $v$ for which the event $S(v)$ occurred. This is crucial for us, as the probability of the event $S(v)$ happening is something we can bound efficiently. 

\begin{claim}\label{concludingfrom}
If Waiter already won after $2^{k/6-\frac{k}{\log k}} \, n$ rounds have passed, there must exist at least $\frac{n}{2k}$ vertices $v$ such that the event $S(v)$ has occurred.
\end{claim}

\begin{proof}
Since by Claim \ref{initialclaim}, there can be at most $\frac{n}{2k}$ high degree cliques among $C_1,...,C_{n/k}$, there must be at least $\frac{n}{2k}$ low degree cliques among them. But by Lemma \ref{thirdsees} applied to an arbitrary low degree clique $C_j$, there exists a vertex $v_j$ of $C_j$ for which the event $Y(v_j)$ has occurred. As $C_j$ is a low degree clique, that means the event $S(v_j)$ has occurred as well. Finally, as the cliques $C_1,...,C_{n/k}$ are all disjoint, we conclude there must be at least $\frac{n}{2k}$ such vertices.
\end{proof}

So now, we know we can finish the proof of Theorem \ref{weakerthm} if we can show the following.

\begin{lemma}\label{wclemmakey}
No matter what strategy Waiter uses, after $2^{k/6-\frac{k}{\log k}} \,n$ rounds have passed, we have

$$\mathbb{E}\Big( \sum_{v} 1_{S(v)} \Big) \leq \frac{n}{4k}   .$$
\end{lemma}

In particular, note that Lemma \ref{wclemmakey} establishes that Waiter can win against randomly playing Client within $2^{k/6-\frac{k}{\log k}} \,n$ rounds with probability at most $1/2$.

To start the proof of Lemma \ref{wclemmakey}, we define

$$Y=\lbrace \mathbf{y}=(y_1,...,y_{k-1}): y_i \in \mathbb{N}, 1 \leq y_1 <...<y_{k-1} \leq 2^{k/6-\frac{k}{2\log k}} \rbrace  .$$

Next, for $\mathbf{y} \in Y$ and a vertex $v$, we define the event $T(v,\mathbf{y})$ as follows. Label the vertices in $N_R(v)$ at the time when $2^{k/6-\frac{k}{\log k}}n$ rounds have passed as $w_1,...,w_l$, where the vertices $w_i$ are ordered by the time when the edge $v w_i$ appeared. Then $T(v,\mathbf{y})$ is the event that:

\begin{itemize}
    \item $v, w_{y_1},...,w_{y_{k-1}}$ is a red clique;
    \item moreover, at least $\frac{(k-1)(k-2)}{6}$ edges $w_{y_i} w_{y_j}$ were added after both the edges $v w_{y_i}$ and $v w_{y_j}$ were added.
\end{itemize}

We divide the rest of the proof into three claims.

\begin{claim}\label{indivy}
For any $\textbf{y} \in Y$ and any $v$, we have $\mathbb{P}(T(v,\mathbf{y})) \leq 2^{-\frac{k^2}{6}+k}$, regardless of the strategy of Waiter.
\end{claim}

\begin{proof}
To complete a red clique $v, w_{y_1},...,w_{y_{k-1}}$ in such a way that the event $T(v,\mathbf{y})$ would occur, Waiter has to offer Client at least $\frac{(k-1)(k-2)}{6}$ edges $w_{y_i} w_{y_j}$ at the time when both the edges $v w_{y_i}$ and $v w_{y_j}$ were already added. If at any time Waiter offers Client two such edges at the same time, one of the edges receives a blue colour and then the probability of the clique $v, w_{y_1},...,w_{y_{k-1}}$ to end up all red is zero. Hence, Waiter can only offer Client one such edge at time, always succeeding with probability $1/2$ and independently of the other rounds. This gives the bound 

$$ \mathbb{P}(T(v,\mathbf{y})) \leq 2^{-\frac{(k-1)(k-2)}{6}}<2^{-\frac{k^2}{6}+k}, $$

as required.

Note that it is not an issue to us that throughout, we may not yet know what vertices will the latter ones, like $w_{y_{k-1}}$, be, or how many edges $w_{y_i} w_{y_j}$ precisely will be added after both the edges $v w_{y_i}$ and $v w_{y_j}$ were added. The definition of our event simply guarantees there must be at least $\frac{(k-1)(k-2)}{6}$ such edges, and at any time that any of these is offered, we know that if it is coloured blue, then the event $T(v,\mathbf{y})$ cannot occur anymore.
\end{proof}

\begin{claim}\label{unionbounding}
For any $v$, we have $$ S(v) \subset \bigcup_{\mathbf{y} \in Y} T(v,\mathbf{y}) .$$
\end{claim}

\begin{proof}
For $S(v)$ to occur, $v$ needs to be a low degree vertex. Hence, whatever subset of $k-1$ vertices in $N_R(v)$ sees the event $Y(v)$ happen must consist of the first $2^{k/6-\frac{k}{2\log k}}$ vertices connected in red to $v$, and hence also sees the corresponding event $T(v, \mathbf{y})$ happen for some $\mathbf{y} \in Y$.
\end{proof}

Combining Claim \ref{indivy}, Claim \ref{unionbounding} and the union bound, we obtain the final result that we need.

\begin{claim}\label{probsv}
We have $\mathbb{P}(S(v))<\frac{1}{4k}$ for any $v$.
\end{claim}

\begin{proof}
Using the results above, we have

\begin{align*}
    & \mathbb{P}(S(v))  \leq \mathbb{P}\big( \bigcup_{\mathbf{y} \in Y} T(v,\mathbf{y}) \big) 
      \leq \sum_{\mathbf{y} \in Y} \mathbb{P}(T(v,\mathbf{y})) 
     \leq |Y| \, 2^{-\frac{k^2}{6}+k} \\
    & =   { 2^{k/6-\frac{k}{2\log k}}   \choose k-1} 2^{-\frac{k^2}{6}+k} 
     < 2^{(k/6-\frac{k}{2\log k})k}2^{-\frac{k^2}{6}+k} 
     =2^{k-\frac{k^2}{2 \log k}} <\frac{1}{4k},
\end{align*}
provided $k \geq 100$.
\end{proof}

But now, Lemma \ref{wclemmakey} follows immediately from Claim \ref{probsv}, and hence Theorem \ref{weakerthm} follows as well.

\section{An improved lower bound}\label{imprbound}

In this section, we prove the lower bound in Theorem \ref{wckthm}, with the $o(k)$ term being $\frac{k}{\log k}$ provided $k \geq 10^8$. 

The strategy we will use will be like in the previous section, but this time with a more general definition of the events whose probabilities we are bounding. Roughly, we will use that if after $2^{k/3-\frac{k}{\log k}}n$ rounds, there was a red $k$-clique factor already created, many of the cliques $C$ in this $k$-clique factor will contain a vertex $z_i$ such that, if we denote the vertices of $C$ by $z_1,...,z_k$:

\begin{itemize}
    \item $z_i$ has red degree less than $2^{k/3-\frac{k}{2\log k}}$;
    \item for many pairs $z_j,z_l$, both the vertices $z_j$ and $z_l$ were in the red connected component of $z_i$ in the graph spanned by $z_1,...,z_k$ before the red edge $z_j z_l$ was added (here many means $(\frac{1}{3}-o(1))k^2$).
\end{itemize} 
This improves the bound significantly, though the details get more technical.

As we only change several ingredients in the proof from the previous section while keeping the rest very similar, we omit the proofs of some claims and instead reference the proofs of the analogous claims in section \ref{basicbound}. We also use the same notation for the events playing the same role, in the hope that this will make the connections to the motivating simpler proof in section \ref{basicbound} clearer.

Once again, the strategy of Client will be random - Client always takes the two edges offered and chooses the one to be added to his graph randomly, with equal probabilities for each and independently of the other rounds. The aim will also be as previously - to show that no matter what strategy Waiter uses, after $2^{k/3-\frac{k}{\log k}} \, n$ rounds have passed, she could have won with probability at most $1/2$ against randomly playing Client. This guarantees a desired deterministic strategy of Client to survive $2^{k/3-\frac{k}{\log k}} \, n$ rounds.

As before, we shall colour the edges of Client's graph red and the edges of Waiter's graph blue, and we shall denote by $N_R(v)$ and $N_B(v)$ the red and the blue neighbourhoods of a vertex $v$. 

Assume that $k \geq 10^8$, that $n$ is divisible by $k$ and large enough that Waiter has a winning strategy for the corresponding Waiter-Client $K_k$-factor game, and that $2^{k/3-\frac{k}{\log k}} \, n$ rounds have passed. If Waiter already won at this point, there are disjoint red cliques $C_1,...,C_{n/k}$ partitioning the vertices of the graph.

Once again, call a vertex a \textit{high degree vertex} if it has red degree at least $2^{k/3-\frac{k}{2\log k}}$, and a \textit{low degree vertex} otherwise. And call a clique a \textit{high degree clique} if it contains at least one high degree vertex, and a \textit{low degree clique otherwise}.

\begin{claim}\label{initialclaim2}
Among $C_1,...,C_{n/k}$, there are at most $\frac{n}{2k}$ high degree cliques.
\end{claim}

\begin{proof}
The proof is analogous to the proof of Claim \ref{initialclaim}.
\end{proof}

We once again define several events. For a given vertex $v$, denote by $X(v)$ the event that $v$ is a low degree vertex after $2^{k/3-\frac{k}{\log k}} \, n$ rounds have passed. Denote by $Y(v)$ the event that:

\begin{itemize}
    \item there exists a subset $w_1,...,w_{k-1}$ in $N_R(v)$ after $2^{k/3-\frac{k}{\log k}} \, n$ rounds have passed that forms a red clique;
    \item moreover, none of the vertices $w_1,...,w_{k-1}$ is a high degree vertex;
    \item finally, at least $\frac{k^2}{3}-\frac{k^2}{(\log k)^2}$ edges $w_i w_j$ were added after both $w_i$ and $w_j$ were already in the red connected component of $v$ in the graph spanned by $v,w_1,...,w_{k-1}$.
\end{itemize}

Let $S(v)=X(v) \cap Y(v)$.

The lemma that follows corresponds to Lemma \ref{thirdsees} in the previous section - but as the result is stronger, we need more care proving it.

\begin{lemma}\label{twothirdslemma}
Start with an empty graph on $k$ vertices and keep adding edges in some order until our graph is complete. Then there exists a vertex $v$ such that at least $\frac{k^2}{3}-\frac{k^2}{(\log k)^2}$ edges $w_i w_j$ were added after both $w_i$ and $w_j$ were already in the connected component of $v$.
\end{lemma}

Let us note that up to the smaller order terms, Lemma \ref{twothirdslemma} is tight. Indeed, consider an initially empty graph on $k=2^t$ vertices for some large $t$. First, create a matching in this graph. After this step we have $2^{t-1}$ connected components $A_1,...,A_{2^{t-1}}$, each consisting of two vertices. Next, fill in all the edges between $A_{2i-1}$ and $A_{2i}$ for $i=1,...,2^{t-2}$, leading to $2^{t-2}$ connected components of four vertices each, each forming a clique. Continue in this manner until the end, always halving the number of the connected components and doubling the number of vertices in each component in each step, and making sure that each connected component is a clique at the end of each step. It is not hard to see that for each vertex $v$, $(1+o(1))\frac{k^2}{3}$ edges $w_i w_j$ were added after both $w_i$ and $w_j$ were already in the connected component of $v$.

\begin{proof}[Proof of Lemma \ref{twothirdslemma}]
We will count all the pairs $(v,w_i w_j)$ consisting of the vertex $v$ and the edge $w_i w_j$ with the property that the edge $w_i w_j$ was added after both $w_i$ and $w_j$ were already in the connected component of $v$. If we show there are at least $\frac{k^3}{3}-\frac{k^3}{(\log k)^2}$ such pairs, we are done, as that means some vertex $v_0$ is counted in at least
$\frac{k^2}{3}-\frac{k^2}{(\log k)^2}$
such pairs.

Call an edge \textit{rare} if it is one of the first $4 (\log k)^2$ edges added at some vertex $z$. Clearly, there are at most $4 k (\log k)^2$ rare edges. 

Further, call an edge \textit{connective at $z$} if it has one endpoint at a vertex $z$ and connects a vertex $z$ to a connected component of at least $ 4 (\log k)^2$ vertices that $z$ was previously not connected to. Call an edge \textit{connective} if there exists some vertex $z_0$ such that this edge is connective at $z_0$. Clearly, for any vertex $z$, there can be at most $\frac{k}{4 (\log k)^2}$ edges connective at $z$ used throughout, since no other vertex $z'$ can be in two different connected components that $z$ is connected to by an edge that is connective at $z$. Hence overall we have at most $\frac{k^2}{4 (\log k)^2}$ connective edges over all the vertices.

Consider all the unordered triples of distinct vertices $(v_1,v_2,v_3)$ such that none of the edges $v_1 v_2$, $v_1 v_3$ and $v_2 v_3$ is rare or connective. Then we claim that if the edge $v_a v_b$, $a,b \in \lbrace 1,2,3 \rbrace$, was added first out of the edges $v_1 v_2$, $v_1 v_3$ and $v_2 v_3$ and $c \in \lbrace 1,2,3 \rbrace$ is such that $c \neq a,b$, then the pairs $(v_a, v_b v_c)$ and $(v_b,v_a v_c)$ were counted. Once we show that, we are done, as summing over all the triples and using that we have less than $$4k(\log k)^2+\frac{k^2}{4 (\log k)^2}<\frac{k^2}{3.5 (\log k)^2}$$ rare or connective edges, this gives at least $$2 \Big( { k \choose 3} -k\frac{k^2}{3.5 (\log k)^2} \Big) >\frac{k^3}{3}-\frac{k^3}{ (\log k)^2}$$ such pairs, provided $k \geq 10^8$.

So we are left to show that if the edge $v_1 v_2$ was added before the edges $v_1 v_3$ and $v_2 v_3$ were added, and none of the edges $v_1 v_2$, $v_1 v_3$ and $v_2 v_3$ is rare or connective, then the pair $(v_1, v_2 v_3)$ was counted (we can without loss of generality reduce just to this case, as for the rest we can just relabel as needed). If the pair $(v_1, v_2 v_3)$ was not counted, that would mean either $v_1$ was not in the same component as $v_2$ or $v_1$ was not in the same component as $v_3$ at the time when the edge $v_2 v_3$ was added. As the edge $v_1 v_2$ was added even before, it is clearly impossible that $v_1$ was not in the same component as $v_2$. But if $v_2$ and $v_3$ were not in the same component at the time when the edge $v_2 v_3$ was added, then either this edge would have been one of the first $4 (\log k)^2$ edges added at $v_3$, and hence rare, or it would have connected $v_2$ to a new connected component of at least $4 (\log k)^2$ vertices (as in particular the connected component of $v_3$ contains the entire neighbourhood of $v_3$), and hence it would have been connective at $v_2$. As neither is true by assumption, we know this also could not have happened, $v_2$ and $v_3$ (and hence also $v_1$ and $v_3$) were in the same connected component at the time when the edge $v_2 v_3$ was added, and we are done.
\end{proof}

Having now proven this result, we continue as in the previous section.

\begin{claim}\label{concludingfrom2}
If Waiter already won after $2^{k/3-\frac{k}{\log k}}n$ rounds have passed, there must exist at least $\frac{n}{2k}$ vertices $v$ such that the event $S(v)$ has occurred.
\end{claim}

\begin{proof}
The proof is analogous to the proof of Claim \ref{concludingfrom}.
\end{proof}

So now, once again, we know we can finish the proof if we can show the following lemma.

\begin{lemma}\label{wclemmakey2}
No matter what strategy Waiter uses, after $2^{k/3-\frac{k}{\log k}} \, n$ rounds have passed, we have

$$\mathbb{E} \Big( \sum_{v} 1_{S(v)} \Big) \leq \frac{n}{4k}   .$$
\end{lemma}

This time, we have to define a more involved set of labels which will let us take $S(v)$ as a subset of the union of events whose probabilities we can easily bound. Let

\begin{align*}
Z=\lbrace \mathbf{y}=(y_1,z_1,y_2,z_2,...,z_{k-2},y_{k-1}): y_i,z_i \in \mathbb{N}, \\
1 \leq y_1,...,y_{k-1} \leq 2^{k/3-\frac{k}{2\log k}}, 1 \leq z_i \leq i+1 \rbrace  .
\end{align*}

Enumerate the vertices of our graph as $v_1,...,v_n$. Moreover, for each $v_i$, enumerate its red neighbours after $2^{k/3-\frac{k}{\log k}} \, n$ rounds have passed as $w_{i,1},w_{i,2},...,w_{i,t_i}$, where the labels correspond to the order in which these red edges were added.

Next, for $\mathbf{y} \in Z$, let $T(v_s,\mathbf{y})$ be the following event:

\begin{itemize}
    \item we have a red clique $x_1,...,x_k$, consisting of low degree vertices only, where $x_1=v_s$, $x_2=w_{s,y_1}$ and we always obtain the next vertex as follows: if we have $x_1,...,x_{i+1}$ already chosen, $z_i=m$ (for some $1 \leq m \leq i+1$) and $x_m=v_d$, then $x_{i+2}=w_{d,y_{i+1}}$;
    \item moreover, at least $\frac{k^2}{3}-\frac{k^2}{(\log k)^2}$ edges $x_i x_j$ were added after both $x_i$ and $x_j$ were already in the red connected component of $v_s$ in the graph spanned by $x_1,...,x_k$;
    \item further, if $z_i=m$, then $x_m$ is the first vertex out of $x_1,...,x_{i+1}$ that got connected to $x_{i+2}$;
    \item finally, there is no pair $x_{i_1},x_{i_2}$ with $i_1<i_2$ such that $x_{i_2}$ appeared in the red connected component of $v_s$ in the graph spanned by $x_1,...,x_k$ strictly sooner than $x_{i_1}$ (of course, both could have appeared at the same time, say when a new connected component consisting of several vertices gets added to the connected component of $v_s$ in this graph).
\end{itemize}

As in the previous section, we divide the rest of the proof into three claims.


\begin{claim}\label{indivy2}
For any $\textbf{y} \in Z$, we have $\mathbb{P}(T(v,\mathbf{y})) \leq 2^{-\frac{k^2}{3}+\frac{k^2}{(\log k)^2}}$, regardless of the strategy of Waiter.
\end{claim}

\begin{proof}
The proof is very similar to the proof of Claim \ref{indivy} and hence is omitted. The important thing that makes our argument still work is the fact that once again, we can identify the vertex $x_i$ as soon as it appears in the connected component of $v$ in the graph spanned by the final clique $x_1,...,x_k$ that $v$ is part of. So any edges inside the connected component of $v$ within this clique that are offered and are counted among the at least $\frac{k^2}{3}-\frac{k^2}{(\log k)^2}$ edges crucial for our argument are known to have this role before they are offered.
\end{proof}

Let us note that Claim \ref{indivy2} only establishes an upper bound. For many $\mathbf{y} \in Z$, the event $ T(v,\mathbf{y})$ cannot happen at all. Say if for $i_1<i_2$, we have $z_{i_1}=z_{i_2}$ but $y_{i_1}>y_{i_2}$, then the corresponding event cannot happen, as it is easy to check that it is then impossible to fulfill all the encoding rules.

The following claim is once again very similar to Claim \ref{unionbounding}, though this time bit less immediate.

\begin{claim}\label{unionbounding2}
For any $v$, we have $$ S(v) \subset \bigcup_{\mathbf{y} \in Z} T(v,\mathbf{y}) .$$
\end{claim}

\begin{proof}
It is easy to see that if a red clique $v,w_1,...,w_{k-1}$ (consisting of low degree vertices only) was created, which moreover has the property that at least $\frac{k^2}{3}-\frac{k^2}{(\log k)^2}$ edges $w_i w_j$ were added after both $w_i$ and $w_j$ were already in the connected component of $v$ in the graph spanned by $v,w_1,...,w_{k-1}$, then we can encode this into a suitable event $T(v,\mathbf{y})$. Indeed, just start with $v$ and keep adding the vertices in order in which they appear in its red connected component in the graph spanned by $v,w_1,...,w_{k-1}$, following the rule during our encoding that if $x_m$ is the first vertex out of $x_1,...,x_{i+1}$ that got connected to $x_{i+2}$, then we set $z_i=m$ and choose $y_{i+1}$ accordingly. When the entire new connected component is added in a single round to the connected component of $v$, add its vertices in such an order that at any point, the graph spanned by $x_1,...,x_{i+1}$ at that point in time is connected. As all the vertices $v,w_1,...,w_{k-1}$ are low degree ones (and hence the labels $y_i$ will be in the required range), it is easy to check that such an encoding indeed works.
\end{proof}

Combining Claim \ref{indivy2}, Claim \ref{unionbounding2} and the union bound, we once again obtain the final result.

\begin{claim}\label{probsv2}
We have $\mathbb{P}(S(v))<\frac{1}{4k}$ for any $v$.
\end{claim}

\begin{proof}
Using the results above, we have

\begin{align*}
    & \mathbb{P}(S(v))  \leq \mathbb{P} \big( \bigcup_{\mathbf{y} \in Z} T(v,\mathbf{y}) \big) 
     \leq \sum_{\mathbf{y} \in Z} \mathbb{P}(T(v,\mathbf{y})) \\
    &  \leq |Z| \, 2^{-\frac{k^2}{3}+\frac{k^2}{(\log k)^2}}  \leq k^k 2^{(k/3-\frac{k}{2 \log k})k}2^{-\frac{k^2}{3}+\frac{k^2}{(\log k)^2}} < \frac{1}{4k},
\end{align*}
provided $k \geq 10^8$.
\end{proof}

But now, Lemma \ref{wclemmakey2} follows immediately from Claim \ref{probsv2}, and hence we are also done proving the lower bound in Theorem \ref{wckthm}.

\section{Upper bound}\label{secupbound}

In this section, we prove the upper bound in Theorem \ref{wckthm}. Roughly, we follow the approach of Clemens et al. \cite{clemens2020fast}, who used it for the case $k=3$, i.e. the triangle-factor game, but we spell out the details for the convenience of the reader.

Assume $k \geq 4$ is fixed, $n$ is divisible by $k$ and is large enough (for instance $n \geq 4^{8^k}$ will do). Once again, colour the edges of Client's graph red and the edges of Waiter's graph blue.

The crucial ingredient of our proof is the following algorithm that Waiter has.

\begin{lemma}\label{keyalgo}
Given $l \geq 2$ and $ 2^{l}-1$ vertices $v_1,v_2,...,v_{2^l-1}$ of an initially empty board, Waiter can in at most $ 2^{l}$ rounds of playing on this board create a red clique $w_1,...,w_l$ with $w_1=v_1$, and moreover keep the property that every edge placed so far (whether red or blue) has at least one endpoint of the form $w_i$ for some $1 \leq i \leq l$.
\end{lemma}

\begin{proof}
To do that, Waiter uses the following approach. Waiter sets $$w_1=v_1,S_1=\lbrace v_1,v_2,v_3,...,v_{2^l-1} \rbrace.$$

After, Waiter offers Client one by one pairs of edges $(v_1 v_{2i}, v_1 v_{2i+1})$ for $i=1,...,2^{l-1}-1$. This now gives a subset of $ 2^{l-1}-1$ vertices connected to $v_1$ in red, and we label this subset as $S_2$ and set $w_2$ to be a vertex $v_j$ of the smallest index $j$ in $S_2$ (so in this case, this smallest index $j$ will be $2$ or $3$).

We continue iteratively in this manner. Given $S_j,w_j$ (where $j<l$ and $|S_j|=2^{l+1-j}-1$), Waiter offers Client the edges from $w_j$ to $S_j \setminus \lbrace w_j \rbrace$ in pairs, sets $S_{j+1}$ to be the set of the vertices of $S_j \setminus \lbrace w_j \rbrace$ connected to $w_j$ in red and sets $w_{j+1}$ to be the vertex $v_m$ of the smallest index $m$ in $S_{j+1}$. Note that in particular this ensures that $|S_{j+1}|=2^{l+1-(j+1)}-1$.

Waiter can clearly continue like this until we obtain $S_l,w_l$ and then $w_1,...,w_l$ is our desired clique. 

The number of rounds that had passed before our clique has been created is 

$$ (2^{l-1}-1)+(2^{l-2}-1)+...+(2-1)<2^l  .$$

The property that every edge placed so far has at least one endpoint of the form $w_i$ for some $1 \leq i \leq l$ trivially holds as well.
\end{proof}

Now we are ready to describe the strategy that Waiter will use to guarantee winning within the desired number of rounds. 

Waiter will proceed in three stages. Throughout, Waiter will update a set $F$ of vertices which is initially empty and has the property at any point in time that the vertices of $F$ contain a red $K_k$-factor.

\vspace{3mm}

\textbf{Stage I.}
In the first stage, Waiter creates a red clique $R$ on $8^k$ vertices in the first $2^{8^k}$ rounds and ensures that every edge placed so far has at least one endpoint in $R$ and the set of vertices $S_0$ defined as

$$ S_0=\lbrace v : \,  v \text{ is an endpoint of at least one red or blue edge} \rbrace $$
satisfies $|S_0| < 2^{8^k}$.

To do that, Waiter simply picks an arbitrary set $S_0$ of $2^{8^{k}}-1$ vertices and uses an algorithm from Lemma \ref{keyalgo}. 

Denote by $A$ the vertices of our graph not in $S_0$ and set $B=S_0 \setminus R$ (note that unlike $F$, the sets $A,B,R$ will not be updated further).
After this stage, Waiter still keeps $F=\emptyset$.

\vspace{3mm}

\textbf{Stage II.} In this stage, Waiter keeps picking $2^{k}-1$ vertices at time (with the first vertex $v_1$ always being from $B$ and all the other ones from $A$ until all the vertices of $B$ are used, and after using just the vertices from $A$) and creating a new $K_k$ using the algorithm from Lemma \ref{keyalgo}, insisting as mentioned that the one vertex we have from $B$ is always in the resulting clique (until we run out of the vertices in $B$, which will happen before this stage ends, as we insisted that $n \geq 4^{8^k}$). Whenever Waiter creates such a clique, she puts its $k$ vertices into $F$. Waiter does this until there are less than $2^{k}$ vertices in $A \setminus F$ left.

\vspace{3mm}

\textbf{Stage III.} Now we have some vertices $z_1,...,z_t$ for $0 \leq t < 2^{k}$ in $A \setminus F$ left. One by one, Waiter offers for each $z_i$ pairs of the edges between $z_i$ and the vertices in what is left in $R \setminus F$ to Client until she creates a clique with one vertex $z_i$ and $k-1$ vertices in $R \setminus F$. Then Waiter takes the $k$ vertices of the resulting clique and puts them into $F$. Due to our constraints, we can see Waiter has enough time to do this for all the vertices $z_1,...,z_t$, and what is left of the graph (i.e. not in $F$) after this process is a subset of $R$, which hence also decomposes into $k$-cliques and can be just put into $F$ immediately. Thus, we have created a red $K_k$-factor.

\vspace{3mm}

In total, we see that Waiter has needed at most $2^{8^k}+ 2^k \, \frac{n}{k}+  2^{k}k $ rounds, proving the upper bound in Theorem \ref{wckthm} with $C(k)=2^{8^k}+  2^{k}k$.

\section{Conclusion}\label{secconclu}

The upper bound and the lower bound in Theorem \ref{wckthm} are still very far apart. We suspect that the upper bound is roughly of the correct magnitude, as it is hard to imagine Waiter could come up with a better strategy than some variant of the natural one from section \ref{secupbound}.

As a first step towards proving the tight result for the lower bound, we believe one could perhaps try to improve the lower bound in Theorem \ref{wckthm} to $2^{k/2-o(k)}n$. Indeed, the logic behind our belief that this should be easier than going past this barrier is as follows. In section \ref{basicbound}, we obtained the bound of the form $2^{k/6-o(k)}n=2^{\frac{1}{3}(k/2)-o(k)}n$ as we managed to include approximately one third of the edges in the low degree cliques into our argument. We could then improve this to $2^{k/3-o(k)}n=2^{\frac{2}{3}(k/2)-o(k)}n$ in section \ref{imprbound} since using a more involved argument, we could make use of about two thirds of the edges in the low degree cliques (which, as mentioned in the relevant section, is tight for the particular argument that we use). So if one managed to come up with a different argument using almost all the edges in the low degree cliques (which does not seem inconceivable), then there would be a hope for the bound of the form $2^{k/2-o(k)}n$.

We expect getting past this barrier to be even more difficult and to require a different approach. Nonetheless, we presume it still should be sufficient to use a random play as the strategy of Client in this higher range. 

\section*{Acknowledgements}

The author would like to thank his PhD supervisor Professor B\'ela Bollob\'as for his helpful comments.  

\bibliographystyle{abbrv}
\bibliography{sample}

\end{document}